\documentclass{amsart}
\usepackage{newtxtext}
\usepackage{latexsym}
\usepackage{amsfonts}
\usepackage[all]{xy}
\usepackage{lmodern}
\usepackage{amscd,color}
\usepackage{amsmath}
\usepackage{amssymb}
\usepackage{mathtools} 
\usepackage[pdftex]{hyperref}  
\usepackage[utf8]{inputenc} 
\usepackage[english]{babel}
\usepackage{cancel} 
\usepackage{amssymb, mathrsfs, amsfonts, amsmath}
\usepackage{amsbsy}
\usepackage{amsfonts}
\newtheorem{corollary}{Corollary}

\newtheorem{definition}{Definition}
\newtheorem{lemma}{Lemma}

\newtheorem{remark}{Remark}
\newtheorem{theorem}{Theorem}

\numberwithin{equation}{section}

\begin{document}
	
	\title[Majumdar-Papapetrou class of solutions]{Solution for the Einstein-Maxwell equations invariant under an $(n - 1)$-dimensional group of dilations}

	\author{Benedito Leandro}
	\address{{B. Leandro - Department of Mathematics, University of Brasilia\\
			Brasilia-DF, 70910-900, Brazil.}
		\email{bleandrone@mat.unb.br}
		\thanks{Partially supported by CNPq/Brazil Grant 303157/2022-4, CNPq/Brazil Grant 400078/2025-2, and FAPDF - 00193-00001678/2024-39.}}
	
	\author{Ilton Menezes}
	\address{{I. Menezes - Centro das Ci\^encias Exatas e das Tecnologias, Universidade Federal do Oeste da Bahia, CEP 47808-021, Barreiras, BA, Brazil}
		\email{ilton.menezes@ufob.edu.br}
		\thanks{}}

	\author{Rafael Novais}
	\address{{R. Novais - Instituto Federal de Educaç\~ao, Ci\^encia e Tecnologia, Campus Posse, CEP 73900-000 Posse, GO, Brazil}
		\email{rafael.novais@ifgoiano.edu.br}
		\thanks{}}

	\thanks{Corresponding Author: B. Leandro (bleandrone@mat.unb.br)}

	\keywords{Electrostatic system, Conformal metric, Einstein-Maxwell equations, Electrostatic system, Electrovacuum system, Majumdar-Papapetrou solution, Symmetry groups }  \subjclass[2020]{Primary 53C25; 83C22; 53C21; 53Z05}

	\date{\today}
	
	\dedicatory{}
	
	\begin{abstract}
		We consider an electrostatic system whose spatial factor is conformal to an $n$-dimensional Euclidean space. We provide a complete characterization of the most general ansatz, thereby reducing the associated electrostatic system of partial differential equations to an ordinary differential equation system. We prove that there are only two possibilities: either the cosmological constant is nonzero, in which case the solutions are necessarily invariant under rotations or translations, or the cosmological constant vanishes, and the solutions belong to the Majumdar--Papapetrou class with a degree of freedom associated with an invariant $(n-1)$-dimensional subgroup. As a result, we introduce a new solution to the electrovacuum system in the Majumdar--Papapetrou class that is invariant under an $(n-1)$--dimensional group of dilations.
	\end{abstract}

	\maketitle


	\section{Introduction}
	
	The Einstein--Maxwell equations with a cosmological constant $\Lambda\in\mathbb{R}$ on a Lorentzian manifold $(\widehat{M}^{n+1},\hat g)$ are given by
	\begin{equation}
		\left\{
		\begin{array}{rcll}
			\displaystyle
			\operatorname{Ric}_{\hat g}-\frac{R_{\hat g}}{2}\hat g + \Lambda \hat g
			&=&
			2\left(F\circ F-\frac{1}{4}|F|_{\hat g}^2\hat g\right), \\[0.3cm]
			dF&=&0,\quad \text{and}\quad \operatorname{div}_{\hat g}F=0,
		\end{array}
		\right.
	\end{equation}
	where $F$ denotes the (Faraday) electromagnetic $(0,2)$-tensor and
	$(F\circ F)_{\alpha\beta}=\hat g^{\sigma\gamma}F_{\alpha\sigma}F_{\beta\gamma}$,
	with Greek indices ranging from $1$ to $n+1$.
	
	A static spacetime is a product manifold $\widehat{M}^{n+1}=\mathbb{R}\times M^n$
	endowed with the metric
	\[
	\hat g=-N^2dt^2+\overline g,
	\]
	where $(M^n,\overline g)$ is an oriented $n$-dimensional Riemannian manifold and
	$N$ is a positive smooth function on $M^n$.
	Choosing the electromagnetic field in the form
	$F=NE^\flat\wedge dt$, where $E^\flat$ denotes the $1$-form dual to the electric
	field $E$, the Einstein--Maxwell system reduces to the following definition
	(see \cite{freitas,Lousa} and the references therein).
	
	\begin{definition}\label{def1}
		Let $(M^n,\overline g)$ be an $n$-dimensional Riemannian manifold,
		$E\in\mathcal X(M)$ a tangent vector field, and $N\in C^\infty(M)$ a positive function.
		The Einstein--Maxwell equations with cosmological constant $\Lambda$ for the static
		spacetime associated with $(M^n,\overline g,N,E)$ are given by
		\begin{equation}\label{eq1-def1}
			\Delta_{\overline g}N
			=
			2N\left(\frac{n-2}{n-1}|E|_{\overline g}^2-\frac{\Lambda}{n-1}\right),
		\end{equation}
		\begin{equation}\label{eq2-def1}
			\operatorname{div}_{\overline g}E=0,
			\qquad
			\operatorname d(NE^\flat)=0,
		\end{equation}
		and
		\begin{equation}\label{eq3-def1}
			\nabla_{\overline g}^2N
			=
			N\left(
			\operatorname{Ric}_{\overline g}
			-\frac{2\Lambda}{n-1}\overline g
			+2E^\flat\otimes E^\flat
			-\frac{2}{n-1}|E|_{\overline g}^2\overline g
			\right).
		\end{equation}
		In this case, $(M^n,\overline g,N,E)$ is called an \emph{electrostatic system}.
		If $(M^n,\overline g)$ is complete, then the electrostatic system is also complete.
	\end{definition}
	
	Here $|\cdot|_{\overline g}$, $\Delta_{\overline g}$, and $\nabla^2_{\overline g}$
	denote the norm, Laplacian, and Hessian with respect to $\overline g$, respectively.
	Moreover, $\operatorname{Ric}_{\overline g}$ and $\operatorname{div}_{\overline g}$
	are the Ricci tensor and divergence associated with $\overline g$.
	The function $N$ and the manifold $M^n$ are referred to as the
	\emph{lapse function} (or static potential) and the spatial factor of the static
	Einstein--Maxwell spacetime.
	An electrostatic system with vanishing cosmological constant is called an
	\emph{electrovacuum system}. In particular, when both the cosmological constant and
	the electric field vanish identically, Definition~\ref{def1} reduces to the static
	vacuum Einstein equations.
	
	Contracting~\eqref{eq3-def1} and combining it with~\eqref{eq1-def1}, we obtain
	\begin{equation}\label{1.4}
		R_{\overline g}=2\left(|E|_{\overline g}^2+\Lambda\right),
	\end{equation}
	where $R_{\overline g}$ denotes the scalar curvature of $(M^n,\overline g)$.
	Since $\operatorname d(NE^\flat)=0$, Poincar\'e's lemma implies the local existence
	of a smooth function $\psi$ such that
	\[
	NE^\flat=\operatorname d\psi.
	\]
	Hence,
	\[
	NE=\nabla_{\overline g}\psi,
	\]
	and we refer to $(M^n,\overline g,N,\psi)$ as an electrostatic system with an electric
	potential $\psi$.
	
	The Reissner--Nordström spacetime is one of the most important solutions of
	Definition~\ref{def1} and can be interpreted as a model for a static black hole
	or a star with electric charge $q$ and mass $m$.
	This solution is called \emph{subextremal}, \emph{extremal}, or
	\emph{superextremal} according to whether $m^{2}>q^{2}$, $m^{2}=q^{2}$,
	or $m^{2}<q^{2}$, respectively.
	Another fundamental electrovacuum solution is the Majumdar--Papapetrou (MP)
	solution, which is associated with the extremal Reissner--Nordström spacetime.
	Both the Reissner--Nordström and Majumdar--Papapetrou solutions are among the
	most important and well-studied electrovacuum systems; see
	\cite{Lousa,Ana}.
	
	The Majumdar--Papapetrou solution represents the static equilibrium of an arbitrary number of charged black holes, whose mutual electric repulsion exactly balances their gravitational attraction.
	This remarkable configuration was later understood to arise as a supersymmetric solution of supergravity.
	More recently, it has been shown that this family constitutes the only
	class of BPS black holes in this theory; see \cite{Lucietti} and the
	references therein.
	The case $m>|q|$ has been completely classified by extending the
	ingenious method of Bunting and Masood-ul-Alam to higher dimensions,
	showing that the unique non-trivial asymptotically flat regular solution
	is the Reissner--Nordström black hole
	(cf.~\cite{Bunting,Chrusciel} for the formal classification of the
	asymptotically flat subextremal electrovacuum system).
	
	In this work, we focus on the extremal case, namely $m=|q|$.
	In this context, Kunduri and Lucietti~\cite{KunduriLucietti} proved that
	an asymptotically flat electrovacuum space $(M^n,g,N,\psi)$ satisfies
	$m=|q|$ if and only if
	\begin{eqnarray}\label{GMP1}
		\pm\sqrt{\frac{2(n-2)}{(n-1)}}\,\psi = 1 - N.
	\end{eqnarray}
	This characterization implies that extremal solutions must belong to
	the Majumdar--Papapetrou class, that is, the class of electrovacuum
	spacetimes $(\mathbb{R}\times M^n,\hat g)$ satisfying~\eqref{GMP1} and
	\begin{eqnarray}\label{MPmetric}
		\hat g=-N^{-2}dt^{2}+\overline g,
	\end{eqnarray}
	where the conformal relation $\overline g=\varphi^{2}g$ holds with
	$\varphi=N^{1/(n-2)}$; see
	\cite{KunduriLucietti,Lousa,Lucietti}.
	The multi-centered extremal black hole solution~\cite{Hawking} is given by
	\begin{eqnarray}\label{MPC}
		N = 1 + \sum_{I=1}^{N}\frac{q_I}{r_I^{\,n-2}},
	\end{eqnarray}
	where $r_I=|x-p_I|$ denotes the Euclidean distance to each center
	$p_I\in\mathbb{R}^n$, with Cartesian coordinates
	$x=(x_1,\ldots,x_n)$.
	Lucietti~\cite{Lucietti} proved that the only asymptotically flat
	spacetimes with a suitably regular event horizon within the
	Majumdar--Papapetrou class of higher-dimensional electrovacuum solutions
	are precisely the standard multi-black holes~\eqref{MPC}.
	Nevertheless, this remains an open problem in full generality;
	see~\cite{KunduriLucietti,Lucietti}.
	
	This work introduces a distinct class of non-asymptotically flat
	Majumdar--Papapetrou solutions.
	Our approach is based on the use of an ansatz that reduces the
	electrostatic system of partial differential equations to a system of
	ordinary differential equations, which can then be explicitly solved.
	Such reduction techniques are classical in the study of nonlinear
	geometric partial differential equations and have been successfully
	applied in a variety of contexts, including warped product constructions
	and rotationally symmetric settings, such as the Bryant soliton, one of
	the most prominent examples of Ricci solitons.
	
	In the broader effort to classify Ricci solitons, several different
	ansatzes have been developed over the years.
	In \cite{keti1}, steady Ricci solitons invariant under the action of an
	$(n-1)$-dimensional translation group were obtained; analogous results
	for Yamabe solitons were established in \cite{BeneditoTenenblat}.
	These constructions were inspired by the work of \cite{Tenenblat}, where
	the symmetry groups of the intrinsic generalized wave and sine--Gordon
	equations were characterized and subsequently employed to produce
	explicit solutions.
	It was shown therein that the corresponding symmetry groups consist only
	of translations and dilations.
	
	Subsequently, in \cite{Be}, the authors studied gradient Ricci solitons
	conformal to an $n$-dimensional pseudo-Euclidean space and described the
	most general substitution reducing the associated system of partial
	differential equations to ordinary differential equations.
	As a consequence, the resulting gradient Ricci solitons are invariant
	under either an $(n-1)$-dimensional translation group or the
	pseudo-orthogonal group acting on the underlying space.
	Related results for static metrics were obtained in \cite{Santos}, where
	the $n$-dimensional Schwarzschild solution arises as an application.
	Static solutions of the Einstein equations invariant under translations
	were further investigated in \cite{Be1,Ana}, while translation-invariant
	solutions for quasi-Einstein metrics were derived in \cite{Ernani}.
	More recently, Lie point symmetries were employed in \cite{RO} to
	construct metrics solving the Ricci curvature and Einstein equations
	with prescribed tensors.
	
	The novelty of the present work lies in the existence of a class of
	static solutions to the Einstein--Maxwell equations with vanishing
	cosmological constant for which the group invariant is determined by an
	equation involving the lapse function.
	In particular, this reveals the presence of additional group-invariant
	solutions within the Majumdar--Papapetrou class, including solutions
	invariant under dilations.
	This behavior contrasts with the rigidity typically observed in known
	group-invariant solutions of Einstein-type equations, such as Ricci
	solitons, static vacuum metrics, and quasi-Einstein metrics, which are
	usually invariant only under translations or rotations
	(see \cite{Santos,Be}).
	The existence of static metrics with fundamentally different types of
	invariance is therefore noteworthy.

	
	To prove our main result, namely Theorem~\ref{ansatz}, item~(1), we carefully
	analyze the system given in Definition~\ref{def1}.
	To this end, we use a conformally flat metric to reduce the system to a set
	of ordinary differential equations.
	This reduction involves a lengthy computational process, which we outline
	below.
	By carrying out this analysis, we obtain a new family of solutions to the
	electrovacuum system in the Majumdar--Papapetrou class, which is invariant
	under an $(n-1)$-dimensional group of dilations, see Theorem~\ref{thm3}.
	This phenomenon is particularly interesting from the point of view of
	partial differential equations because, as far as we know, it has not been
	observed before for equations such as Ricci solitons, Yamabe solitons,
	quasi-Einstein metrics, and static vacuum metrics; see
	\cite{keti1,Be1,BeneditoTenenblat,Be,Ernani,Santos} and the references
	therein.
	
	We now introduce some fundamental notation before stating our main results.
	The functions $\varphi$, $N$, and $\psi$ define a solution to the
	electrostatic system $(\Omega,\overline g,N,\psi)$, where
	$\overline g = g/\varphi^{2}$ is a metric conformal to the Euclidean metric
	$g$, with a conformal factor $\varphi$.
	Throughout the manuscript, $\Omega\subseteq\mathbb{R}^{n}$ denotes an open
	subset on which the functions $\varphi$, $N$, and $\psi$ are defined.
	For convenience, we adopt the conformal change
	$\overline g = g/\varphi^{2}$ instead of $\overline g = \varphi^{2} g$, as
	commonly used in the Majumdar--Papapetrou class, noting that this choice
	does not affect the results, and it is merely a convenience.

	\begin{theorem}\label{ansatz}
		Let $(\mathbb{R}^{n},g)$ be the Euclidean space with Cartesian coordinates
		$(x_1,\ldots,x_n)$ and Euclidean metric $g$.
		Then there exists a smooth function
		$\xi=\xi(x_1,\ldots,x_n)$ such that
		$(\Omega,\overline g,N(\xi),\psi(\xi))$, where
		$\overline g=g/\varphi^{2}(\xi)$, is a solution of the electrostatic system
		if and only if one of the following conditions holds:
		\begin{enumerate}
			\item
			the solution belongs to the Majumdar--Papapetrou class, satisfies
			\eqref{GMP1} and \eqref{MPmetric} with $\Lambda=0$, and
			\begin{eqnarray}\label{eq:xi-condition}
				-\,\frac{d}{d\xi}
				\log\!\left(
				-\frac{d}{d\xi}\left(N^{-1}\right)
				\right)
				=
				\frac{\sum_{k=1}^{n}\partial^{2}_{x_k}\xi}
				{\sum_{k=1}^{n}\left(\partial_{x_k}\xi\right)^2};
			\end{eqnarray}
			
			\item
			the solution does not belong to the Majumdar--Papapetrou class, and
			\begin{eqnarray}\label{eq:ansatz}
				\xi(x_1,\ldots,x_n)
				=
				\Gamma\!\left(
				\sum_{k=1}^{n}\tau x_k^{2}+\gamma_k x_k+\theta_k
				\right).
			\end{eqnarray}
		\end{enumerate}
		Here $\tau$, $\gamma_k$, and $\theta_k$ are real constants, and
		$\Gamma$ is a smooth real function.
	\end{theorem}
	
	The case $\tau=0$ in item~(2) generalizes the solutions obtained
	in~\cite{Ana} and yields solutions invariant under translations.
	When $\tau\neq 0$, the solutions are spherically symmetric, a case already
	covered in~\cite{Bunting,Chrusciel}.
	Therefore, in the remainder of this work we focus on item~(1) of
	Theorem~\ref{ansatz}.
	In particular, this result allows one to choose an arbitrary smooth
	function $\xi$, provided that equation~\eqref{eq:xi-condition} is
	solvable.

	We show that there exists a Majumdar-Papapetrou solution of the
	electrovacuum equations that is invariant under dilations (not asymptotically flat \cite{KunduriLucietti,Lucietti}). Furthermore, we will show that the classical MP solution can be recovered from Theorem~\ref{ansatz}--Item (1); see Corollary~\ref{MPclassic} below.

	\begin{theorem}\label{thm3}
		Let $\left(\mathbb{R}^{n},g\right)$ be the Euclidean space with Cartesian coordinates $x=(x_1,...\, ,x_n)$ and the Euclidean metric $g$.      Consider $\Omega\subset\mathbb{R}^n$ an open set such that the lapse function $N$ is defined. Here, $1\leq m_1\leq m_2\leq n$, $k\neq0$, $k_1$ are constants. Moreover, 
		\begin{eqnarray*}
			\eta=\sum_{k=1}^{m_2}b_k^2,\quad \theta=-2\sum_{k=1}^{m_1}a_kb_k,\quad \delta=\sum_{k=1}^{m_1}a_k^2,\quad \sum_{i=1}^{m_1}a_i\neq 0,\quad\mbox{and}\quad \prod_{j=1}^{m_2}b_j\neq0,
		\end{eqnarray*}
		where $a_i,\,b_j\in\mathbb{R}$. Then, $$(\Omega,\,\overline{g},\,N,\,\psi),$$ is a solution for the electrovacuum system with
		\begin{eqnarray*}
			\qquad\overline{g}=\dfrac{g}{\varphi^{2}},\,\qquad\pm\sqrt{\dfrac{2(n-2)}{(n-1)}}\psi = 1 - N,\qquad \varphi = N^{1/(n-2)},
		\end{eqnarray*}
		and
		\begin{eqnarray*}
			\dfrac{1}{N(x)}= k_1 + \dfrac{2k}{ \sqrt{4\eta\delta - \theta^2}}\arctan\left( \dfrac{\theta\displaystyle\sum_{j=1}^{m_2}b_jx_j + 2\eta\displaystyle\sum_{i=1}^{m_1}a_ix_i}{\sqrt{4\eta\delta - \theta^2}\displaystyle\sum_{j=1}^{m_2}b_jx_j} \right),
		\end{eqnarray*}
		where $4\eta\delta - \theta^2>0$.
	\end{theorem}
	
	As a consequence of the theorem above, we obtain a new example of a static
	space-time solution of the Einstein-Maxwell equations with a vanishing
	cosmological constant, defined on
	$\widehat{M}^{n+1}=\mathbb{R}\times\Omega$ and endowed with the metric
	\[
	\widehat{g}=-N^{2}dt^{2}+\frac{g}{\varphi^{2}}.
	\]
	
	We now describe more precisely the domain $\Omega$ on which the lapse
	function is well defined.

	\begin{remark}[The domain $\Omega$ of the lapse function]\label{rem:lapse-positive}
		The lapse function obtained in Theorem~\ref{thm3} depends on the invariant
		\[
		\xi(x)=\frac{M(x)}{P(x)},\qquad 
		M(x)=\sum_{i=1}^{m_1} a_i x_i,
		\quad 
		P(x)=\sum_{j=1}^{m_2} b_j x_j.
		\]
		In particular, the hypotheses $\sum_{i=1}^{m_1} a_i\neq 0$ and $\prod_{j=1}^{m_2} b_j\neq 0$
		ensure that $\delta>0$ and $\eta>0$.

		By the Cauchy--Schwarz inequality,
		\[
		4\eta\delta-\theta^2
		=
		4\left[
		\Big(\sum_{j=1}^{m_2} b_j^2\Big)\Big(\sum_{i=1}^{m_1} a_i^2\Big)
		-
		\Big(\sum_{i=1}^{m_1} a_i b_i\Big)^2
		\right]
		\ge 0.
		\]
		Moreover, equality holds if and only if the vectors
		$(a_1,\dots,a_{m_1})$ and $(b_1,\dots,b_{m_2})$ are linearly dependent which it is possible to happen if $m_1=m_2$, excluding the case $4\eta\delta-\theta^2=0$. Since, in this case, $\xi(x)$ is a constant function. Consequently, the natural analytic domain of $\xi$ (and hence of $N$) is
		\[
		\Omega_0=\{x\in\mathbb{R}^n:\;P(x)\neq 0\},
		\]
		since the invariant $\xi$ is not defined on the hyperplane $\{P=0\}$.
		
		Considering $4\eta\delta-\theta^2>0$, the lapse function is given by the arctangent function, which is well defined for all $\xi\in\mathbb{R}$.
		Remember, $\arctan(t)\in(-\pi/2,\pi/2)$ for all $t\in\mathbb{R}$, and it follows that
		\[
		k_1-\frac{k\pi}{\sqrt{4\eta\delta-\theta^2}}
		<
		N^{-1}(x)
		<
		k_1+\frac{k\pi}{\sqrt{4\eta\delta-\theta^2}},
		\qquad
		\forall\,x\in\Omega_0.
		\]
		In particular, $N^{-1}$ is uniformly bounded on $\Omega_0$. Therefore,
		\[
		\{x\in\Omega_0:\;N(x)=0\}=\varnothing.
		\]
		In this case, the domain of the solution is the open subset
		\[
		\Omega=\{x\in\Omega_0:\;N(x)>0\},
		\]
		selected solely by the geometric requirement of positivity of the lapse function.
	\end{remark}
	
	\begin{remark}[Uniform equivalence and completeness of Theorem \ref{thm3}]
		Considering that
		\[
		4\eta\delta-\theta^2>0
		\]
		and the lapse function given by the arctangent expression.
		Since $\arctan(t)\in(-\pi/2,\pi/2)$ for all $t\in\mathbb{R}$, we have
		\[
		k_1-\frac{k\pi}{\sqrt{4\eta\delta-\theta^2}}
		<
		N^{-1}(x)
		<
		k_1+\frac{k\pi}{\sqrt{4\eta\delta-\theta^2}},
		\qquad
		\forall x\in\Omega,
		\]
		where $\Omega$ is any open set on which $N$ is defined and positive.
		Defining
		\[
		A=k_1-\frac{k\pi}{\sqrt{4\eta\delta-\theta^2}},
		\qquad
		B=k_1+\frac{k\pi}{\sqrt{4\eta\delta-\theta^2}},
		\]
		we note that $B>0$.
		If, in addition, $A>0$, then $N^{-1}$ is uniformly bounded away from zero.
		
		In this case, the conformal metric
		\[
		\overline g = N^{-\frac{2}{n-2}}\, g
		\]
		is uniformly equivalent to the Euclidean metric $g$, i.e., there exist positive
		constants $c_1,c_2$ such that
		\[
		c_1\, g \;\le\; \overline g \;\le\; c_2\, g
		\quad \text{on } \Omega.
		\]
		Consequently, $(\Omega,g)$ is complete if and only if $(\Omega,\overline g)$ is complete.
	\end{remark}


	\section{Background}
	\label{Sec2}
	
	In what follows, we adopt the following convention for the derivatives of a
	function $F=F(\xi)$, where $\xi:\mathbb{R}^{n}\to\mathbb{R}$ is a smooth
	function depending on the Cartesian coordinates
	$x=(x_1,\ldots,x_n)$.

	\begin{center}
		$\dfrac{d F}{d\xi}=F'$,\,\quad$\dfrac{\partial F}{\partial x_i}=F_{,i}$\quad\mbox{and}\quad
		$\dfrac{\partial^2F}{\partial x_i\partial x_j}=F_{,ij}$.
	\end{center}

	\begin{lemma}\label{cont:lem1}
		Let $(\mathbb{R}^n,g)$, $n\geq 3$, be the Euclidean space with Cartesian coordinates $(x_1,..., x_n)$ and the Euclidean metric $g$. Consider smooth functions $\varphi,\, \psi$ and $N$ and a metric $\overline{g}=g/\varphi^{2}$ such that $\left(\Omega,\,\overline{g},\, N,\,\psi\right)$, $\Omega\subseteq\mathbb{R}^{n}$, is a solution for the electrostatic system. Then, the functions $\varphi,\, \psi$ and $N$ must satisfy
		\begin{equation}\label{23}
			\displaystyle\sum_{k=1}^{n}\left[2(n-1) N\varphi_{,kk} -n(n-1)\frac{N}{\varphi}(\varphi_{,k})^2 - 2\frac{\varphi}{N} (\psi_{,k})^2\right]= \frac{2N}{\varphi}\Lambda.
		\end{equation}
	\end{lemma}
	\begin{proof}[Proof of Lemma~\ref{cont:lem1}]
		From equation~\eqref{1.4}, it follows that
		\begin{equation*}
			R_{\overline g}
			=
			2\left(
			\frac{|\nabla_{\overline g}\psi|^{2}}{N^{2}}+\Lambda
			\right),
		\end{equation*}
		where $R_{\overline g}$ and $\nabla_{\overline g}$ denote the scalar
		curvature and the gradient with respect to $\overline g$, respectively.
		It is well known that the scalar curvature of $\overline g$
		(cf.~\cite{khunel1988}) is given by
		\begin{equation*}
			R_{\overline g}
			=
			(n-1)\left(
			2\varphi\,\Delta_{g}\varphi
			-
			n|\nabla_{g}\varphi|^{2}
			\right).
		\end{equation*}
		Moreover, if $\psi$ is a smooth function on $(\Omega,\overline g)$, then
		\begin{equation*}
			|\nabla_{\overline g}\psi|^{2}
			=
			\varphi^{2}\sum_{k=1}^{n}(\psi_{,k})^{2}.
		\end{equation*}
		Combining the identities above with equation~\eqref{1.4}, we obtain
		\begin{equation*}
			(n-1)\left(
			2\varphi\,\Delta_{g}\varphi
			-
			n|\nabla_{g}\varphi|^{2}
			\right)
			=
			2\left(
			\frac{|\nabla_{\overline g}\psi|^{2}}{N^{2}}+\Lambda
			\right).
		\end{equation*}
		
		In Cartesian coordinates, this identity becomes
		\begin{equation*}
			\sum_{k=1}^{n}
			\left[
			2(n-1)N\varphi_{,kk}
			-
			n(n-1)\frac{N}{\varphi}(\varphi_{,k})^{2}
			-
			2\frac{\varphi}{N}(\psi_{,k})^{2}
			\right]
			=
			\frac{2N}{\varphi}\Lambda,
		\end{equation*}
		which concludes the proof.
	\end{proof}

	
	Next, we reduce the electrostatic system given in Definition~\ref{def1}
	to a system of partial differential equations involving the functions
	$\psi$, $\varphi$, and $N$, together with their partial derivatives.

	\begin{theorem}\label{theo1}
		Let $\left(\mathbb{R}^{n},g\right)$, $\, n\geq 3$, be the Euclidean space with Cartesian coordinates $x=\left(x_{1},...\, ,x_{n}\right)$ and Euclidean metric components $g$. Then, there exists a metric $\overline{g}=g/ \varphi^{2}$ such that $\left(\Omega,\,\overline{g},\,N,\,\psi\right)$, $\Omega\subseteq\mathbb{R}^{n}$, is a solution for the electrostatic system if and only if the smooth functions $\varphi$, $\psi$ and $N$ satisfy 
		
		\begin{equation}\label{eq1-theo1}
			(n-2)N\varphi_{,ij}-\varphi N_{,ij}-\varphi_{,i}N_{,j}-\varphi_{,j}N_{,i}+2\frac{\varphi}{N}\psi_{,i}\psi_{,j}=0,\hspace{0,8cm} \mathrm{for} \hspace{0,2cm} i\neq j;
		\end{equation}
		and for each $i$
		\begin{multline}\label{eq2-theo1}
			\varphi\left[(n-2)N\varphi_{,ii}-\varphi N_{,ii}-2\varphi_{,i}N_{,i}+2\frac{\varphi}{N}\left(\psi_{,i}\right)^2 \right]\\
			+\sum_{k=1}^{n}\left[\varphi\varphi_{,kk}N+\varphi\varphi_{,k}N_{,k}-(n-1)N\left(\varphi_{,k}\right)^2-\frac{2}{(n-1)N}\varphi^2\left(\psi_{,k}\right)^2\right]=\frac{2\Lambda}{n-1}N;
		\end{multline}
		
		\begin{equation}\label{eq3-theo1}
			\sum_{k=1}^{n}\left[N\varphi \psi_{,kk}-(n-2)N\varphi_{,k}\psi_{,k}-\varphi\psi_{,k}N_{,k}\right]=0;
		\end{equation}
		
		\begin{equation}\label{eq4-theo1}
			\displaystyle\sum_{k=1}^{n}\left[\varphi^2 NN_{,kk} -(n-2)\varphi\varphi_{,k}NN_{,k}-\frac{2(n-2)}{n-1}\varphi^2 \left(\psi_{,k}\right)^2\right]=-\frac{2\Lambda}{n-1} N^2.
		\end{equation}
	\end{theorem}

	\begin{proof}[Proof of Theorem~\ref{theo1}]
		Recall the expressions for the Ricci tensor and the scalar curvature of a
		conformal metric of the form $\overline g=g/\varphi^{2}$
		(cf.~\cite{khunel1988}):
		\begin{equation*}
			\operatorname{Ric}_{\overline g}
			=
			\frac{1}{\varphi^{2}}
			\Bigl(
			(n-2)\varphi\,\nabla^{2}_{g}\varphi
			+
			\bigl[
			\varphi\Delta_{g}\varphi-(n-1)|\nabla_{g}\varphi|^{2}
			\bigr]g
			\Bigr)
		\end{equation*}
		and
		\begin{equation*}
			R_{\overline g}
			=
			(n-1)\bigl(
			2\varphi\Delta_{g}\varphi
			-
			n|\nabla_{g}\varphi|^{2}
			\bigr).
		\end{equation*}
		
		From equation~\eqref{eq3-def1}, we obtain
		\begin{equation*}
			N\operatorname{Ric}_{\overline g}
			=
			\nabla^{2}_{\overline g}N
			-
			2\frac{\nabla_{\overline g}\psi\otimes\nabla_{\overline g}\psi}{N}
			+
			\frac{2}{n-1}
			\left(
			\frac{|\nabla_{\overline g}\psi|^{2}}{N}
			+
			N\Lambda
			\right)\overline g,
		\end{equation*}
		which is equivalent to
		\begin{multline}\label{A}
			(n-2)N\varphi(\nabla^{2}_{g}\varphi)_{ij}
			+
			N\bigl[
			\varphi\Delta_{g}\varphi
			-
			(n-1)|\nabla_{g}\varphi|^{2}
			\bigr]\delta_{ij} \\
			=
			\varphi^{2}(\nabla^{2}_{\overline g}N)_{ij}
			+
			\frac{2}{n-1}
			\left(
			\frac{|\nabla_{\overline g}\psi|^{2}}{N}
			+
			N\Lambda
			\right)\delta_{ij}
			-
			\frac{2\varphi^{2}}{N}
			(\nabla_{\overline g}\psi\otimes\nabla_{\overline g}\psi)_{ij}.
		\end{multline}
		
		The Hessian of $N$ with respect to $\overline g$ is given by
		\begin{equation*}
			(\nabla^{2}_{\overline g}N)_{ij}
			=
			N_{,ij}
			-
			\sum_{k=1}^{n}\overline\Gamma^{k}_{ij}N_{,k},
		\end{equation*}
		where $\overline\Gamma^{k}_{ij}$ denote the Christoffel symbols of
		$\overline g$.
		For distinct indices $i,j,k$, one has
		\begin{equation}\label{cristofel}
			\overline\Gamma^{k}_{ij}=0,
			\qquad
			\overline\Gamma^{i}_{ij}=-\frac{\varphi_{,j}}{\varphi},
			\qquad
			\overline\Gamma^{k}_{ii}=\frac{\varphi_{,k}}{\varphi},
			\qquad
			\overline\Gamma^{i}_{ii}=-\frac{\varphi_{,i}}{\varphi}.
		\end{equation}
		Therefore,
		\begin{equation}\label{10}
			(\nabla^{2}_{\overline g}N)_{ij}
			=
			N_{,ij}
			+
			\frac{\varphi_{,j}N_{,i}}{\varphi}
			+
			\frac{\varphi_{,i}N_{,j}}{\varphi},
			\qquad i\neq j,
		\end{equation}
		and
		\begin{equation}\label{11}
			(\nabla^{2}_{\overline g}N)_{ii}
			=
			N_{,ii}
			+
			2\frac{\varphi_{,i}N_{,i}}{\varphi}
			-
			\frac{1}{\varphi}\sum_{k=1}^{n}\varphi_{,k}N_{,k},
			\qquad \text{for each } i.
		\end{equation}
		
		Moreover,
		\begin{equation}\label{14}
			|\nabla_{g}\varphi|^{2}
			=
			\sum_{k=1}^{n}(\varphi_{,k})^{2},
			\qquad
			\Delta_{g}\varphi
			=
			\sum_{k=1}^{n}\varphi_{,kk},
		\end{equation}
		and
		\begin{equation}\label{014}
			(\nabla^{2}_{g}\varphi)_{ij}
			=
			\varphi_{,ij},
			\qquad
			(\nabla_{\overline g}\psi\otimes\nabla_{\overline g}\psi)_{ij}
			=
			\psi_{,i}\psi_{,j}.
		\end{equation}
		
		Assume first that $i\neq j$ in~\eqref{A}.
		Then
		\begin{equation}\label{15}
			(n-2)N(\nabla^{2}_{g}\varphi)_{ij}
			-
			\varphi(\nabla^{2}_{\overline g}N)_{ij}
			+
			2\frac{\varphi}{N}\psi_{,i}\psi_{,j}
			=
			0.
		\end{equation}
		Substituting~\eqref{10} and~\eqref{014} into~\eqref{15}, we obtain
		\begin{equation*}
			(n-2)N\varphi_{,ij}
			-
			\varphi N_{,ij}
			-
			\varphi_{,j}N_{,i}
			-
			\varphi_{,i}N_{,j}
			+
			2\frac{\varphi}{N}\psi_{,i}\psi_{,j}
			=
			0.
		\end{equation*}
		
		If instead $i=j$ in~\eqref{A}, then~\eqref{014} yields
		\begin{multline}\label{13}
			(n-2)N\varphi(\nabla^{2}_{g}\varphi)_{ii}
			+
			N\bigl[
			\varphi\Delta_{g}\varphi
			-
			(n-1)|\nabla_{g}\varphi|^{2}
			\bigr] \\
			=
			\varphi^{2}(\nabla^{2}_{\overline g}N)_{ii}
			+
			\frac{2}{n-1}
			\left(
			\frac{|\nabla_{\overline g}\psi|^{2}}{N}
			+
			N\Lambda
			\right)
			-
			\frac{2\varphi^{2}}{N}(\psi_{,i})^{2}.
		\end{multline}
		Combining~\eqref{11}, \eqref{14}, and~\eqref{014} with~\eqref{13}, we obtain
		\begin{multline*}
			\varphi\Bigl[
			(n-2)N\varphi_{,ii}
			-
			\varphi N_{,ii}
			-
			2\varphi_{,i}N_{,i}
			+
			2\frac{\varphi}{N}(\psi_{,i})^{2}
			\Bigr] \\
			+
			\sum_{k=1}^{n}
			\Bigl[
			\varphi\varphi_{,kk}N
			+
			\varphi\varphi_{,k}N_{,k}
			-
			(n-1)N(\varphi_{,k})^{2}
			-
			\frac{2\varphi^{2}}{(n-1)N}(\psi_{,k})^{2}
			-
			\frac{2\Lambda}{(n-1)n}N
			\Bigr]
			=
			0.
		\end{multline*}
		
		Equation~\eqref{eq2-def1}, namely
		$\operatorname{div}_{\overline g}(\nabla_{\overline g}\psi/N)=0$,
		is equivalent to
		\begin{equation}\label{16}
			N\Delta_{\overline g}\psi
			-
			\overline g(\nabla_{\overline g}\psi,\nabla_{\overline g}N)
			=
			0.
		\end{equation}
		A direct computation gives
		\begin{equation}\label{18}
			\overline g(\nabla_{\overline g}\psi,\nabla_{\overline g}N)
			=
			\varphi^{2}\sum_{i=1}^{n}\psi_{,i}N_{,i}.
		\end{equation}
		Since the Laplacian of a smooth function $F$ with respect to $\overline g$
		is given by
		\begin{equation*}
			\Delta_{\overline g}F
			=
			\sum_{i=1}^{n}
			\left[
			\varphi^{2}F_{,ii}
			-
			(n-2)\varphi\varphi_{,i}F_{,i}
			\right],
		\end{equation*}
		it follows from~\eqref{16} and~\eqref{18} that
		\begin{equation*}
			\sum_{k=1}^{n}
			\bigl\{
			N\varphi\psi_{,kk}
			-
			(n-2)N\varphi_{,k}\psi_{,k}
			-
			\varphi\psi_{,k}N_{,k}
			\bigr\}
			=
			0.
		\end{equation*}
		
		Finally, equation~\eqref{eq1-def1} yields
		\begin{equation*}
			\Delta_{\overline g}N
			=
			2N\left(
			\frac{n-2}{n-1}|E|^{2}
			-
			\frac{1}{n-1}\Lambda
			\right),
		\end{equation*}
		where $E=\nabla_{\overline g}\psi/N$.
		Therefore,
		\begin{equation*}
			\sum_{k=1}^{n}
			\left[
			\varphi^{2}NN_{,kk}
			-
			(n-2)\varphi\varphi_{,k}NN_{,k}
			-
			\frac{2(n-2)}{n-1}\varphi^{2}(\psi_{,k})^{2}
			\right]
			=
			-\frac{2\Lambda}{n-1}N^{2}.
		\end{equation*}
		
		The converse implication follows by a straightforward computation.
	\end{proof}


	\section{Proof of the Main Result}

	The following result adapts techniques developed in \cite{Be} and
	\cite{Santos} to the electrostatic system.
	Its purpose is to characterize the most general smooth functions
	$\xi:\mathbb{R}^{n}\to\mathbb{R}$, with $\xi=\xi(x_1,\ldots,x_n)$, for which
	the partial differential equations
	\eqref{eq1-theo1}, \eqref{eq2-theo1}, \eqref{eq3-theo1}, and
	\eqref{eq4-theo1} can be reduced to a system of ordinary differential
	equations.


	\begin{proof}[{\bf Proof of Theorem \ref{ansatz}}]
		Assume that $\varphi(\xi)$, $\psi(\xi)$, and $N(\xi)$ are functions of $\xi$. Then, for instance,
		\begin{equation}\label{d1}
			\displaystyle\varphi_{,i}=\varphi'\xi_{,i};\hspace{0.5cm}\varphi_{,ii}=\varphi''\left(\xi_{,i}\right)^2+\varphi'\xi_{,ii};\hspace{0.5cm}\varphi_{,ij}=\varphi''\xi_{,i}\xi_{,j}+\varphi'\xi_{,ij}.
		\end{equation}
		Thus, from \eqref{eq1-theo1}, 
		\begin{multline}\label{eq1thm2}
			\displaystyle\left[(n-2)\varphi''N-\varphi N''-2\varphi'N'+2\frac{\varphi}{N}\left(\psi'\right)^2\right]\xi_{,i}\xi_{,j}\\ +\left[(n-2)\varphi'N-\varphi N'\right]\xi_{,ij}=0.
		\end{multline}

		The identity \eqref{eq1thm2} must be analyzed in two distinct scenarios. Unlike \cite{Be,Santos}, where $(n-2)\varphi'N-\varphi N'\neq0$ is a trivial fact, it is well-known that the Majumdar-Papapetrou multi-centered extremal solution for the electrovacuum system satisfies $(n-2)\varphi'N-\varphi N'=0$; see \eqref{MPmetric}. Therefore, we will consider these two cases separately.  
		
		First, consider $(n-2)\varphi'N-\varphi N'=0$. Since $\xi_{,i}\xi_{,j}\neq 0$ is a subset of $\mathbb{R}^n\setminus\{0\}$, for some $i\neq j$, from \eqref{eq1thm2},  $(n-2)\varphi''N-\varphi N''-2\varphi'N'+2\dfrac{\varphi}{N}\left(\psi'\right)^2=0$. Taking the derivative of $(n-2)\varphi'N-\varphi N'=0$ yields
		\begin{eqnarray}\label{derivative}
			\displaystyle (n-2)N\varphi''-\varphi N''-2N'\varphi'=-(n-1)\varphi' N'.
		\end{eqnarray}
		Combining $(n-2)\varphi''N-\varphi N''-2\varphi'N'+2\dfrac{\varphi}{N}\left(\psi'\right)^2=0$ with the above identity, we conclude that 
		\begin{eqnarray}\label{eq2thm2}
			\displaystyle \frac{\varphi'}{\varphi}=\frac{1}{n-2}\frac{N'}{N}, \hspace{0,5cm}\text{and}\hspace{0,5cm}\left(\psi'\right)^2=\frac{n-1}{2(n-2)}\left(N'\right)^2.
		\end{eqnarray}
		And this is strong evidence that we are dealing with the Majumdar-Papapetrou class of solutions.

		Moving on, from \eqref{eq2-theo1} and \eqref{eq2thm2}, we have
		\begin{multline}\label{eq3thm2}
			\left[\varphi\varphi''N+\varphi\varphi'N'-(n-1)N(\varphi')^2-\frac{2}{(n-1)N}\varphi^2(\psi')^2\right]\left|\nabla_{g}\xi\right|^2\\ +\varphi\left[-(n-1)N'\varphi'+\frac{2\varphi}{N}\left(\psi'\right)^2\right]\left(\xi_{,i}\right)^2
			+\varphi\varphi'N\Delta_{g}\xi=\frac{2\Lambda}{n-1} N.
		\end{multline}
		Substituting the equation \eqref{derivative} into $(n-2)\varphi''N-\varphi N''-2\varphi'N'+2\dfrac{\varphi}{N}\left(\psi'\right)^2=0$ leads us to 
		\begin{eqnarray}\label{eq4thm2}
			\displaystyle N'\varphi'-\frac{2\varphi}{(n-1)N}\left(\psi'\right)^2=0,
		\end{eqnarray}
		and
		\begin{multline*}
			\varphi\varphi''N+\varphi\varphi'N'-(n-1)N(\varphi')^2-\frac{2}{(n-1)N}\varphi^2(\psi')^2\\
			=\varphi\varphi''N-(n-1)N(\varphi')^2+\varphi\varphi'N'-\frac{2}{(n-1)N}\varphi^2(\psi')^2\\
			=N\left[\varphi\varphi''-(n-1)\left(\varphi'\right)^2\right]+\varphi\left(\varphi'N'-\frac{2}{(n-1)N}\varphi(\psi')^2\right).
		\end{multline*}
		From \eqref{eq4thm2}, we get 
		\begin{eqnarray}
			\varphi\varphi''N+\varphi\varphi'N'-(n-1)N(\varphi')^2-\frac{2}{(n-1)N}\varphi^2(\psi')^2=N\left[\varphi\varphi''-(n-1)\left(\varphi'\right)^2\right].\nonumber
		\end{eqnarray}
		In this case, \eqref{eq3thm2} can be rewritten as
		\begin{eqnarray}\label{eq5thm2}
			\displaystyle\left[\varphi\varphi''-(n-1)\left(\varphi'\right)^2\right]\left|\nabla_{g}\xi\right|^2+\varphi\varphi'\Delta_g\xi=\frac{2\Lambda}{n-1}.
		\end{eqnarray}
		
		On the other hand, by Lemma \ref{cont:lem1}, we can verify that  
		\begin{eqnarray*}
			\displaystyle\left[2(n-1)\varphi''-n(n-1)\frac{\left(\varphi'\right)^2}{\varphi}-\frac{2\varphi}{N^2}\left(\psi'\right)^2\right]\left|\nabla_{g}\xi\right|^2+2(n-1)\varphi'\Delta_g\xi=\frac{2\Lambda}{\varphi}.
		\end{eqnarray*}
		From \eqref{eq2thm2}, we obtain
		\begin{eqnarray}\label{eq6thm2}
			\displaystyle\left[\varphi\varphi''-(n-1)\left(\varphi'\right)^2\right]\left|\nabla_{g}\xi\right|^2+\varphi\varphi'\Delta_g\xi=\frac{\Lambda}{n-1}.
		\end{eqnarray}
		Combining \eqref{eq5thm2} and \eqref{eq6thm2} yields $$\Lambda=0.$$
		This is also expected, since the static Majumdar-Papapetrou class of solutions has a null cosmological constant.

		From equation \eqref{eq1thm2}, we have $$(n-2)\varphi'' N - \varphi N'' - 2 \varphi' N' + 2 \frac{\varphi}{N} \left(\psi'\right)^2 = 0.$$ Moreover, from \eqref{eq2thm2} we can see that this equation is trivially satisfied. Combining \eqref{eq2-theo1} and \eqref{d1}, we obtain
		\begin{eqnarray}
			\left[\varphi\varphi''N+\varphi\varphi'N'-(n-1)\left(\varphi'\right)^2N-\frac{2}{(n-1)N}\varphi^2\left(\psi'\right)^2\right]\sum_{k=1}^{n}\left(\xi_{,k}\right)^2+\varphi\varphi'N\sum_{k=1}^{n}\xi_{,kk}=0. \nonumber
		\end{eqnarray}
		Similarly, from \eqref{eq3-theo1} and \eqref{d1}, we will see that
		\begin{eqnarray}
			\left[\varphi N\psi''-(n-2)\varphi'N\psi'-\varphi N'\psi'\right]\sum_{k=1}^{n}\left(\xi_{,k}\right)^2+\varphi N\psi'\sum_{k=1}^{n}\xi_{,kk}=0, \nonumber
		\end{eqnarray}
		Finally, from \eqref{eq4-theo1} and \eqref{d1},
		\begin{eqnarray}
			\left[\varphi N''-(n-2)\varphi'N'-\frac{2(n-2)}{(n-1)N}\varphi\left(\psi'\right)^2\right]\sum_{k=1}^{n}\left(\xi_{,k}\right)^2+\varphi N'\sum_{k=1}^{n}\xi_{,kk}=0.\nonumber
		\end{eqnarray}
		
		Therefore, we have the system 
		\begin{itemize}
			\item [(1)] $ \displaystyle \left[\varphi\varphi''N+\varphi\varphi'N'-(n-1)\left(\varphi'\right)^2N-\frac{2}{(n-1)N}\varphi^2\left(\psi'\right)^2\right]\vert\nabla_{g}\xi\vert^2+\varphi\varphi'N\Delta_g\xi=0$,
			\vspace{0.3cm}
			\item [(2)] $\displaystyle\left[\varphi N\psi''-(n-2)\varphi'N\psi'-\varphi N'\psi'\right]\vert\nabla_{g}\xi\vert^2+\varphi N\psi'\Delta_g\xi=0$,
			\vspace{0.3cm}
			\item [(3)] $\displaystyle\left[\varphi N''-(n-2)\varphi'N'-\frac{2(n-2)}{(n-1)N}\varphi\left(\psi'\right)^2\right]\vert\nabla_{g}\xi\vert^2+\varphi N'\Delta_g\xi=0$.
		\end{itemize}
		We claim that (1), (2), and (3) are equivalent. In fact, from 
		\begin{eqnarray*}\label{remember-eq2thm2}
			\displaystyle \frac{\varphi'}{\varphi}=\frac{1}{(n-2)}\frac{N'}{N}, \hspace{0,5cm}\text{and}\hspace{0,5cm}\left(\psi'\right)^2=\frac{(n-1)}{2(n-2)}\left(N'\right)^2
		\end{eqnarray*}
		we can see that (1), (2), and (3) holds if and only if
		\begin{eqnarray}
			\displaystyle\frac{N''}{N'}-2\frac{N'}{N}=-\frac{\Delta_g\xi}{\vert\nabla_{g}\xi\vert^2}.
			\nonumber
		\end{eqnarray}
		

		In fact, consider that
		\begin{eqnarray}\label{eq2-eqv}
			\frac{\varphi'}{\varphi}=\frac{1}{n-2}\frac{N'}{N},
			\qquad
			(\psi')^{2}=\frac{n-1}{2(n-2)}(N')^{2}.
		\end{eqnarray}
		Taking the derivative of $(n-2)\varphi' N-\varphi N'=0$ yields
		\begin{eqnarray}\label{eq1-eqv}
			(n-2)N\varphi''-\varphi N''-2N'\varphi'=-(n-1)\varphi' N'.
		\end{eqnarray}
		Note that from a straightforward computation using (1), 
		\begin{multline*}
			\left[\varphi\varphi''N+\varphi\varphi' N'-(n-1)(\varphi')^{2}N
			-\frac{2}{(n-1)N}\varphi^{2}(\varphi')^{2}
			\right]\vert\nabla_{g}\xi\vert^{2}+\varphi\varphi' N\Delta_g\xi\\
			=\frac{\varphi^{2}}{n-2}\Bigg\{\Bigg[
			(n-2)\frac{\varphi''}{\varphi}N+(n-2)\frac{\varphi'}{\varphi}N'
			-\frac{(n-1)}{(n-2)}\left((n-2)\dfrac{\varphi'}{\varphi}\right)^{2}N
			-\dfrac{2(n-2)}{(n-1)}N(\varphi')^{2}
			\Bigg]\vert\nabla_{g}\xi\vert^{2}\\
			\hspace{4cm}
			+(n-2)\dfrac{\varphi'}{\varphi}N\Delta_g\xi
			\Bigg\}.
		\end{multline*}
		Combining the above equation with \eqref{eq2-eqv} and \eqref{eq1-eqv} leads us to
		\begin{multline*}
			\left[\varphi\varphi''N+\varphi\varphi' N'-(n-1)(\varphi')^{2}N-\frac{2}{(n-1)N}\varphi^{2}(\psi')^{2}\right]\vert\nabla_{g}\xi\vert^{2}
			+\varphi\varphi' N\Delta_g\xi\\
			=\dfrac{\varphi^{2}N'}{n-2}\left\{\left[\frac{N''}{N'}-2\frac{N'}{N}\right]\vert\nabla_{g}\xi\vert^{2}
			+\Delta_g\xi\right\}.
		\end{multline*}
		Thus,
		\begin{eqnarray}
			\left[\varphi\varphi''N+\varphi\varphi' N'-(n-1)(\varphi')^{2}N-\frac{2}{(n-1)N}\varphi^{2}(\psi')^{2}\right]\vert\nabla_{g}\xi\vert^{2}
			+\varphi\varphi' N\Delta_g\xi=0.\nonumber
		\end{eqnarray}
		if and only if 
		\begin{eqnarray}
			\frac{N''}{N'}-2\frac{N'}{N}=-\frac{\Delta_g\xi}{\vert\nabla_{g}\xi\vert^2}.\nonumber
		\end{eqnarray}

		Let us consider the equation (2); note that $(\psi')^{2}=\dfrac{n-1}{2(n-2)}(N')^{2}$. consequently, $\psi'=AN'$ and $\psi''=AN''$, where $A=\pm\sqrt{\dfrac{n-1}{2(n-2)}}$. So, 
		\begin{multline*}
			\displaystyle\left[\varphi N\psi''-(n-2)\varphi'N\psi'-\varphi N'\psi'\right]\vert\nabla_{g}\xi\vert^2+\varphi N\psi'\Delta_g\xi\\
			= \left[\varphi NAN''-A(n-2)\varphi'NN'-\varphi A(N')^2\right]\vert\nabla_{g}\xi\vert^2+\varphi ANN'\Delta_g\xi,
		\end{multline*}
		i.e.,
		\begin{multline*}
			\displaystyle\left[\varphi N\psi''-(n-2)\varphi'N\psi'-\varphi N'\psi'\right]\vert\nabla_{g}\xi\vert^2+\varphi N\psi'\Delta_g\xi\\
			=A\varphi NN'\left\{\left[\frac{N''}{N'}-(n-2)\frac{\varphi'}{\varphi}-\frac{N'}{N}\right]\vert\nabla_{g}\xi\vert^2+\Delta_g\xi\right\}.
		\end{multline*}
		Therefore,
		\begin{eqnarray}
			\displaystyle\left[\varphi N\psi''-(n-2)\varphi'N\psi'-\varphi N'\psi'\right]\vert\nabla_{g}\xi\vert^2+\varphi N\psi'\Delta_g\xi=0.\nonumber
		\end{eqnarray} 
		if and only if 
		\begin{eqnarray}
			\frac{N''}{N'}-2\frac{N'}{N}=-\frac{\Delta_g\xi}{\vert\nabla_{g}\xi\vert^2},
		\end{eqnarray}
		where we used that $(n-2)\dfrac{\varphi'}{\varphi}=\dfrac{N'}{N}$.
		
		Finally, for equation (3), note that 
		\begin{multline*}
			\displaystyle\left[\varphi N''-(n-2)\varphi'N'-\frac{2(n-2)}{(n-1)N}\varphi\left(\psi'\right)^2\right]\vert\nabla_{g}\xi\vert^2+\varphi N'\Delta_g\xi\\
			= \varphi N'\left\{\left[\frac{N''}{N'}-(n-2)\frac{\varphi'}{\varphi}-\frac{2(n-2)}{(n-1)N}\frac{\left(\psi'\right)^2}{N'}\right]\vert\nabla_{g}\xi\vert^2+\Delta_g\xi\right\}.
		\end{multline*}
		By equation \eqref{eq2-eqv}, we can see that
		\begin{eqnarray*}
			\displaystyle\left[\varphi N''-(n-2)\varphi'N'-\frac{2(n-2)}{(n-1)N}\varphi\left(\psi'\right)^2\right]\vert\nabla_{g}\xi\vert^2+\varphi N'\Delta_g\xi=\varphi N'\left\{\left[\frac{N''}{N'}-2\frac{N'}{N}\right]\vert\nabla_{g}\xi\vert^2+\Delta_g\xi\right\}.
		\end{eqnarray*}
		It follows that
		\begin{eqnarray}
			\displaystyle\left[\varphi N''-(n-2)\varphi'N'-\frac{2(n-2)}{(n-1)N}\varphi\left(\psi'\right)^2\right]\vert\nabla_{g}\xi\vert^2+\varphi N'\Delta_g\xi=0\nonumber
		\end{eqnarray}
		if and only if 
		\begin{eqnarray*}
			\frac{N''}{N'}-2\frac{N'}{N}=-\frac{\Delta_g\xi}{\vert\nabla_{g}\xi\vert^2}.
		\end{eqnarray*}

		Now, consider $\xi_{,i}\xi_{,j}\neq 0$ in an open subset $\mathbb{R}^n\setminus\{0\}$, for some $i\neq j$, such that $(n-2)\varphi'N-\varphi N'\neq0$. From \eqref{eq1thm2}, we can see that 
		\begin{equation}\label{4}
			\frac{\xi_{,ij}}{\xi_{,i}\xi_{,j}}=H(\xi),
		\end{equation}
		for same smooth function $H(\xi) = 	\dfrac{\left[(n-2)\varphi''N-\varphi N''-2\varphi'N'+2\frac{\varphi}{N}\left(\psi'\right)^2\right]}{\left[(n-2)\varphi'N-\varphi N'\right]}$. By integration, we can infer
		\begin{equation}
			\xi_{,i}=exp\left\{\int H(\xi)d\xi+H_i\left(\tilde{x_j}\right)\right\},
		\end{equation}
		where the variable $x_j$ is arbitrary and $H_i\left(\tilde{x_j}\right)$ does not depend on the variable $x_j$, i.e., $\tilde{x_j}=\left(x_1, ..., x_{j-1},x_{j+1}, ...,x_n\right)$. This process can be considered for all $i\neq j$. Writing  $L\left(\xi\right)=exp\left\{\int H\left(\xi\right) d\xi\right\}$ and $L_i\left(x_i\right)=exp\left\{H_{i}\left(\tilde{x_j}\right)\right\}$, we have
		\begin{equation}\label{ansatz:eq18}
			\xi_{,i}=L\left(\xi\right)L_i\left(x_i\right).
		\end{equation}
		
		Multiplying $-\dfrac{\varphi}{n}$ in Lemma \ref{cont:lem1} yields
		\begin{eqnarray}
			-\displaystyle\sum_{k=1}^{n}\left[\frac{2(n-1)}{n} N\varphi\varphi_{,kk} -(n-1)N(\varphi_{,k})^2 - \frac{2\varphi^2}{nN} (\psi_{,k})^2\right]= -\frac{2N}{n}\Lambda.\nonumber
		\end{eqnarray}	
		Combining the above identity with \eqref{eq2-theo1} leads us to
		\begin{multline}\label{eq1-poof-red}
			\varphi\left[(n-2)N\varphi_{,ii}-\varphi N_{,ii}-2\varphi_{,i}N_{,i}+2\frac{\varphi}{N}\left(\psi_{,i}\right)^2 \right]\\
			+\sum_{k=1}^{n}\left[\frac{(2-n)}{n}N\varphi\varphi_{,kk}+\varphi\varphi_{,k}N_{,k}-\frac{2}{(n-1)n}\frac{\varphi^2}{N}\left(\psi_{,k}\right)^2\right]=\frac{2\Lambda}{n(n-1)}N.
		\end{multline}
		
		Multiple \eqref{eq4-theo1} by $\dfrac{1}{nN}$, i.e., 
		\begin{equation}\label{eq2-poof-red}
			\displaystyle\sum_{k=1}^{n}\left[\frac{\varphi^2 N_{,kk}}{n} -\frac{(n-2)}{n}\varphi\varphi_{,k}N_{,k}-\frac{2(n-2)}{n(n-1)}\frac{\varphi^2}{N} \left(\psi_{,k}\right)^2\right]=-\frac{2\Lambda}{n(n-1)} N.
		\end{equation}	
		Combining \eqref{eq1-poof-red} and \eqref{eq2-poof-red} to obtain
		\begin{multline}\label{eq3-poof-red}
			\left[(n-2)N\varphi_{,ii}-\varphi N_{,ii}-2\varphi_{,i}N_{,i}+2\frac{\varphi}{N}\left(\psi_{,i}\right)^2 \right]\nonumber\\
			-\frac{1}{n}\sum_{k=1}^{n}\left[(n-2)N\varphi_{,kk}-\varphi N_{,kk}-2\varphi_{,k}N_{,k}+2\frac{\varphi}{N}\left(\psi_{,k}\right)^2\right]=0.
		\end{multline}
		Using \eqref{d1}, the above equation becomes
		\begin{multline*}
			\displaystyle\left[(n-2)\varphi''N-\varphi N''-2\varphi'N'+2\frac{\varphi}{N}\left(\psi'\right)^2\right]\left(\xi_{,i}^2-\frac{1}{n}\sum_{k=1}^{n}\xi_{,k}^2\right)\nonumber\\
			+\left[(n-2)\varphi'N-\varphi N'\right]\left(\xi_{,i i}-\frac{1}{n}\sum_{k=1}^{n}\xi_{,k k}\right)=0.
		\end{multline*}

		Considering $H(\xi)$ in \eqref{4}, we conclude that
		\begin{equation}\label{ansatz:eq16}
			H(\xi)\left(n\xi_{,i}^2-\sum_{k=1}^n\xi_{,k}^2\right)=\left(n\xi_{,ii}-\sum_{k=1}^n\xi_{,kk}\right),
		\end{equation}
		consequently,
		\begin{equation}\label{19}
			n\left(\xi_{,i}exp\left\{-\int Hd\xi\right\}\right)_{,i}=\sum_{k=1}^n\left(\xi_{,k}exp\left\{-\int Hd\xi\right\}\right)_{,k}.
		\end{equation}
		
		Combining \eqref{ansatz:eq18} and \eqref{19}, we obtain 
		\begin{equation*}
			\sum_{k=1}^n\left(L_k(x_k)L(\xi)L^{-1}(\xi)\right)_{,k}=n\left(L_i(x_i)L(\xi)L^{-1}(\xi)\right)_{,i}.
		\end{equation*}
		Therefore,
		\begin{equation*}
			\sum_{k=1}^nL'_k=nL'_i.
		\end{equation*}
		Since, the above identity holds for any $i$ we have $nL'_i = nL'_j$, where $i\neq j$, which implies that $L_i\left(x_i\right)= 2\tau x_i+ \gamma_i$ for every $i$, where $\tau, \gamma_i\in \mathbb{R}$. Hence, for each $i\neq j$, from \eqref{ansatz:eq18} we obtain 
		\begin{equation*}
			\frac{\xi_{,i}}{L_i\left(x_i\right)}=L\left(\xi\right),
		\end{equation*}
		i.e., 
		\begin{equation*}
			\frac{\xi_{,i}}{2\tau x_i+ \gamma_i}=\frac{\xi_{,j}}{2 \tau x_i+ \gamma_i}.
		\end{equation*}
		The characteristic of the above equation implies the result.

		The reciprocal of this result is a straightforward computation and similar to the one provided by \cite{Be}. We will sketch the proof for completeness.

		Suppose now that $\xi=\sum_{k=1}^{n}U_k(x_k)$, where $U_k(x_k)=\tau x_k^2+\gamma_kx_k+\theta_k$. Thus,
		\begin{equation*}
			\psi_{,x_i}=\psi'U_i', \hspace{0,5cm}\psi_{,x_ix_j}=\psi''U_i'U_j', \hspace{0,5cm}\psi_{,x_ix_i}=\psi''(U_i')^2+\psi'U_i''.
		\end{equation*}
		The same holds for the functions $\varphi$ and $N$. From equation \eqref{eq1-theo1}, we have
		\begin{eqnarray}
			\left[(n-2)\varphi''N-\varphi N''-2\varphi'N'+2\frac{\varphi}{N} \left(\psi'\right)^2\right]U_i'U_j'=0.\nonumber 
		\end{eqnarray}
		Since there exist $U_i'U_j'\neq0$ in an open subset $\mathbb{R}^n$, 
		\begin{eqnarray}\label{eq1thm1}
			(n-2)\varphi''N-\varphi N''-2\varphi'N'+2\frac{\varphi}{N} \left(\psi'\right)^2=0.
		\end{eqnarray}
		Moreover, for each $i$, the equation \eqref{eq2-theo1} gives
		\begin{multline}\label{eq2thm1}
			\varphi\left[(n-2)\varphi''N-\varphi N''-2\varphi'N'+2\frac{\varphi}{N} \left(\psi'\right)^2\right](U_i')^2+\varphi\left[(n-2)\varphi'N-\varphi N'\right]U_i''\\
			+\left[\varphi\varphi''N-(n-1)\left(\varphi'\right)^2N+\varphi\varphi'N'-\frac{2\varphi^2}{(n-1)N}\left(\psi'\right)^2\right]\sum_{k=1}^{n}(U_k')^2\\ +\varphi\varphi'N\sum_{k=1}^{n}U_k''=\frac{2\Lambda}{n-1}N.
		\end{multline}
		Combining \eqref{eq1thm1}, \eqref{eq2thm1}, $\sum_{k=1}^{n}\left(U_k'\right)^2=4\tau\xi+\beta$, and $\sum_{k=1}^{n}U_k''=2n\tau$, where $\beta=\sum_{k=1}^{n}\left(\gamma_k^2-4\tau\theta_k\right)$, we can rewrite \eqref{eq2thm1} in the following manner: 
		
		\begin{multline}\label{eq3thm1}
			\left[\varphi\varphi''N-(n-1)\left(\varphi'\right)^2N+\varphi\varphi'N'-\frac{2\varphi^2}{(n-1)N}\left(\psi'\right)^2\right]\left(4\tau\xi+\beta\right)\\
			+ 	2\tau\varphi\left[2(n-1)\varphi'N-\varphi N'\right]=\frac{2\Lambda}{n-1}N.
		\end{multline}
		
		From equation \eqref{eq3-theo1}, we obtain
		\begin{eqnarray}\label{eq4thm1}
			2n\tau\varphi\psi'N+\left[\varphi\psi''N-(n-2)\varphi'\psi'N-\varphi\psi'N'\right]\left(4\tau\xi+\beta\right)=0.\nonumber
		\end{eqnarray}
		Finally, equation \eqref{eq4-theo1} yields
		\begin{multline}\label{eq5thm1}
			2n\tau\varphi NN'+\left[\varphi NN''-(n-2)\varphi'NN'-\frac{2(n-2)}{n-1}\varphi\left(\psi'\right)^2\right]\left(4\tau\xi+\beta\right)\\
			=-\frac{2\Lambda}{n-1}\varphi N^2.
		\end{multline}
	\end{proof}

	
	



	In this section, we will show that the classical multi-centered Majumdar-Papapetrou solution can be recovered from Theorem \ref{ansatz}. Moreover, we will provide the proof of Theorem \ref{thm3}, i.e., an example of the Einstein-Maxwell static space that is invariant under an $(n-1)$-dimensional group of dilations.
	

	\begin{corollary}\label{MPclassic}
		Let $\left(\mathbb{R}^{n},\,g\right)$, $n\geq 3$, be the Euclidean space with Cartesian coordinates $x=\left(x_{1},...,x_{n}\right)$ and Euclidean metric $g$. Let $\Omega\subset \mathbb{R}^{n}\setminus\{c^{k}_{1},\ldots,c^{k}_{l}\}$, where $c^{k}_{l} = (c_{l}^{1},\,\ldots,\,c_{l}^{n})\in\mathbb{R}^n$, be an open subset in which the smooth functions $\varphi, \psi$ and $N$ are defined. Then, there exists a metric $\overline{g}=g/\varphi^2$ such that $\left(\Omega,\,\overline{g},\,N,\,\psi\right)$ is a solution for the electrovacuum system in which
		
		\begin{eqnarray}\label{rot}
			\begin{cases}
				\displaystyle N(x)^{-1}= {\sum_{l=1}^{s} \dfrac{k\lambda_l}{r_l^{n-2}}-k_1},\\\\
				\varphi = N^{1/(n-2)},\\\\[1mm]
				\pm\sqrt{\dfrac{2(n-2)}{(n-1)}}\psi = 1 - N,
			\end{cases}
		\end{eqnarray}
		where $r_l^2=\sum_{k=1}^{n}  (x_k - c_l^k)^2$ for each center $c_{l}^{k} = (c_{l}^{1},\,\ldots,\,c_{l}^{n})\in\mathbb{R}^n$. Here, $l\in\{1,\,\ldots,\,s\}.$

	\end{corollary}
	\begin{proof}[{\bf Proof of Corollary \ref{MPclassic}}]
		Consider 
		\begin{eqnarray}
			\xi = -\sum_{l=1}^{s} \lambda_l\left(\sum_{k=1}^{n}  (x_k - c_l^k)^2\right)^{\frac{2-n}{2}},\nonumber
		\end{eqnarray}
		where $c_{l}^{k} = (c_{l}^{1},\,\ldots,\,c_{l}^{n})$ represents the $l$-th multi-centered black hole. The first derivative of the above identity leads us to
		\begin{eqnarray}
			\xi_{,s} = -(2-n)\sum_{k=1}^{n}\lambda_l(x_s - c_l^s)\left[ \sum_{l=1}^{n}(x_k-c_l^k)^2\right]^{-\frac{n}{2}}.\nonumber
		\end{eqnarray}
		Computing the second derivative, we get
		\begin{eqnarray}
			\xi_{,ss} = -(2-n)\sum_{l=1}^{n}\lambda_l\left[\left( \sum_{k=1}^{n}(x_k-c_l^k)^2\right)^{-\frac{n}{2}}-n(x_s-c_l^s)^2\left( \sum_{k=1}^{n}(x_k-c_l^k)^2\right)^{-\frac{n+2}{2}}\right].\nonumber
		\end{eqnarray}
		
		Since $\Delta_{g}\xi=\sum_{s=1}^n\xi_{,ss}$, we have 
		\begin{multline*}
			\Delta_{g}\xi=\sum_s\xi_{,ss}=\sum_{s=1}^N\left(-(2-n)\sum_{k=1}^{n}\lambda_l(x_s - c_l^s)\left[ \sum_{l=1}^{n}(x_k-c_l^k)^2\right]^{-\frac{n}{2}}\right)\\
			=-(2-n)\sum_{l=1}^{n}\lambda_l\left[n\left( \sum_{k=1}^{n}(x_k-c_l^k)^2\right)^{-\frac{n}{2}}-n\sum_{s=1}^n(x_s-c_l^s)^2\left( \sum_{k=1}^{n}(x_k-c_l^k)^2\right)^{-\frac{n+2}{2}}\right] \\ = 0,
		\end{multline*}
		consequently, $\Delta_{g}\xi=0$, i.e., $\xi$ is a harmonic function. From the first item in Theorem \ref{ansatz}, we have
		\begin{eqnarray}
			\displaystyle\frac{N''}{N'}-2\frac{N'}{N}=0.
			\nonumber
		\end{eqnarray}
		Hence, 
		\begin{eqnarray}
			\displaystyle N(\xi)= -\frac{1}{k\xi+k_1},\nonumber
		\end{eqnarray}
		where $k$ and $k_1$ are constants.
		
		Finally, we have
		\begin{eqnarray}
			\displaystyle N(\xi)= \dfrac{1}{\sum_{l=1}^{s} \dfrac{k\lambda_l}{r_l^{n-2}}-k_1},\nonumber
		\end{eqnarray}
		where $r_l^2=\sum_{k=1}^{n}  (x_k - c_l^k)^2$. 
		Taking $k_1 = -1$ and $k=1$ we get the classical Majumdar-Papapetrou solution. Hence, $k\lambda_l$ represents the charge (or mass) of the black holes.
	\end{proof}

	Let us provide solutions of the electrovacuum system \((\Omega, \overline{g}, N, \psi)\) which are invariant under an $(n-1)$-dimensional group of dilations subject to the additional condition that it is in the Majumdar-Papapetrou class, i.e., satisfying Theorem \ref{ansatz} - (1).

	\begin{theorem}\label{sera}[Theorem \ref{thm3}]
		Let $(\mathbb{R}^n,g)$, $n\ge 3$, be the Euclidean space with Cartesian coordinates
		$x=(x_1,\dots,x_n)$ and the Euclidean metric $g$.
		Let $\Omega\subset\mathbb{R}^n$ be an open set. Take the integers $1\le m_1\le m_2\le n$ and the real constants
		$a_i$ $(1\le i\le m_1)$, $b_j$ $(1\le j\le m_2)$ satisfying
		\[
		\sum_{i=1}^{m_1} a_i \neq 0,
		\qquad
		\prod_{j=1}^{m_2} b_j \neq 0.
		\]
		Define the dilation--invariant function
		\[
		\xi(x)=\dfrac{\sum_{i=1}^{m_1} a_i x_i}{\sum_{j=1}^{m_2} b_j x_j}.
		\]
		Set
		\[
		\eta=\sum_{j=1}^{m_2} b_j^2,
		\qquad
		\theta=-2\sum_{i=1}^{m_1} a_i b_i,
		\qquad
		\delta=\sum_{i=1}^{m_1} a_i^2,
		\]
		and arbitrary constants $k\neq0$ and $k_1\in\mathbb{R}$, define the lapse
		function $N$ as follows:
		\begin{eqnarray*} 
			\dfrac{1}{N(x)}
			=
			k_1
			+
			\frac{2k}{\sqrt{4\eta\delta-\theta^2}}
			\arctan\!\left(
			\frac{2\eta\,\xi(x)+\theta}{\sqrt{4\eta\delta-\theta^2}}
			\right),
		\end{eqnarray*}
		where $4\eta\delta-\theta^2>0$.
		Define also
		\[
		\varphi = N^{1/(n-2)},
		\qquad
		\psi=\pm\sqrt{\frac{n-1}{2(n-2)}}\,(1-N),\qquad\mbox{and}\qquad\overline{g}=g/\varphi^{2}.
		\]

		Then, $(\Omega,\,\overline{g},\,N,\,\psi)$ is a solution to the electrovacuum Einstein--Maxwell system with a vanishing cosmological constant. Moreover, the associated static spacetime belongs to the Majumdar--Papapetrou class and is invariant under an $(n-1)$--dimensional group of dilations.
	\end{theorem}
	
	\begin{proof}[{\bf Proof of Theorem \ref{sera}}]
		
		Consider $\mathbb{R}^{n}$ as the Euclidean space with Cartesian coordinates $x=\left(x_{1},...\, ,x_{n}\right)$, and take the dilation invariant
		\begin{eqnarray}\label{dil}
			\xi(x)=\frac{\sum_{i=1}^{m_1}a_ix_i}{\sum_{j=1}^{m_2}b_jx_j}, \hspace{0,5cm} 1\leq m_1\leq m_2\leq n,
		\end{eqnarray}
		where $a_i$ and $b_j$ are real constants satisfying 
		\begin{eqnarray}\label{dil2}
			\sum_{i=1}^{m_1}a_i\neq 0,\quad\mbox{and}\quad \prod_{j=1}^{m_2}b_j\neq0,
		\end{eqnarray}
		see \cite{RO,Tenenblat} and the references therein.
		
		Define  
		\begin{eqnarray}
			M(x) = \sum_{i=1}^{m_1} a_i x_i, \qquad P(x) = \sum_{j=1}^{m_2} b_j x_j, \qquad 1 \le m_1 \le m_2 \le n,\nonumber
		\end{eqnarray}  
		so that the variable \(\xi\) can be expressed in the following manner  
		\begin{eqnarray*}
			\xi(x) = \frac{M(x)}{P(x)}.
		\end{eqnarray*}  
		
		The partial derivative of \(\xi\) with respect to the variable \(x_k\) is given by  
		\begin{eqnarray*}
			\xi_{,k}= \frac{M_{,k} \, P - M \, P_{,k}}{P^2}.
		\end{eqnarray*}

		\medskip
		\noindent \textbf{First:} \(1 \leq k \leq m_1\).  
		In this case, we have $M_{,k} = a_k$ and $P_{,k} = b_k$. Therefore,
		\begin{eqnarray}\label{partial-derivative-eq1}
			\xi_{,k}=\frac{a_k - b_k\xi}{P}, \hspace{0,5cm} \xi_{,kk} = -\frac{2\,b_k}{P}\xi_{,k}.
		\end{eqnarray} 
		
		\medskip
		\noindent \textbf{Second:} \(m_1 < k \leq m_2\).  
		Here, $M_{,k} = 0$ and $P_{,k} = b_k$. Hence,
		\begin{eqnarray}\label{partial-derivative-eq2}
			\xi_{,k} = -\frac{b_k}{P}\xi, \hspace{0,5cm} \xi_{,kk} = \frac{2b^2_k}{P^2}\xi.
		\end{eqnarray}

		Note that $\vert\nabla_{g}\xi\vert^2=\sum_{k=1}^{n}\xi_{,k}^2$. Since we are considering the dilation invariant, it follows that $\xi_{,k}=\dfrac{a_k - b_k\xi}{P}$ for \(1 \leq k \leq m_1\) and $\xi_{,k} = -\dfrac{b_k}{P}\xi$ for \(m_1 < k \leq m_2\). Thus, $\vert\nabla_{g}\xi\vert^2=\sum_{k=1}^{m_1}\dfrac{\left(a_k - b_k\xi\right)^2}{P^2}+\sum_{k=m_1+1}^{m_2}\dfrac{b_k^2\xi^2}{P^2}$. Consequently, $P^2\vert\nabla_{g}\xi\vert^2=\sum_{k=1}^{m_1}\left(a_k-b_k\xi\right)^2+\sum_{k=m_1+1}^{m_2}b_k^2\xi^2$, and hence 
		\begin{multline}\label{d3}
			\frac{d}{d\xi}\left(P^2\vert\nabla_{g}\xi\vert^2\right)=\dfrac{d}{d\xi}\left(\sum_{k=1}^{m_1}\left(a_k-b_k\xi\right)^2+\sum_{k=m_1+1}^{m_2}b_k^2\xi^2\right)\\
			=-2\sum_{k=1}^{m_1}b_k\left(a_k-b_k\xi\right)+2\sum_{k=m_1+1}^{m_2}b_k^2\xi.
		\end{multline}

		On the other hand, since $\Delta_{{g}}\xi=\sum_{k=1}^n\xi_{,kk}$, we get
		\begin{eqnarray} 
			\Delta_{g}\xi=-2\sum_{k=1}^{m_1}\frac{\,b_k}{P}\xi_{,k}+\sum_{k=m_1+1}^{m_2}\frac{2b^2_k}{P^2}\xi.\nonumber
		\end{eqnarray}
		Moreover, $\xi_{,k}=\dfrac{a_k - b_k\xi}{P}$ for $k\leq m_1$. So,  
		\begin{eqnarray} 
			\Delta_{g}\xi=-2\sum_{k=1}^{m_1}\frac{b_k\left(a_k - b_k\xi\right)}{P^2}+\sum_{k=m_1+1}^{m_2}\frac{2b^2_k}{P^2}\xi.\nonumber
		\end{eqnarray}
		Consequently, 
		\begin{eqnarray}\label{lap} 
			P^2\Delta_{g}\xi=-2\sum_{k=1}^{m_1}b_k\left(a_k - b_k\xi\right)+2\sum_{k=m_1+1}^{m_2}b^2_k\xi.
		\end{eqnarray}
		Combining the equations \eqref{d3} with \eqref{lap}, we have the fundamental relation given by the dilation invariant:
		\begin{eqnarray*}
			\dfrac{1}{P^2}\dfrac{d}{d\xi}\left(P^2\vert\nabla_{g}\xi\vert^2\right)=\Delta_g \xi.
		\end{eqnarray*}

		We know from \eqref{eq:xi-condition} that the lapse function satisfies the equation
		\begin{eqnarray}\label{eqboaaaa}
			\displaystyle\frac{N''}{N'}-2\frac{N'}{N}=-\frac{\Delta_g\xi}{\vert\nabla_{g}\xi\vert^2}.
		\end{eqnarray}
		Therefore, 
		\begin{eqnarray*}
			\dfrac{d}{d\xi}\log\left(-\dfrac{d}{d\xi}\left(N^{-1}\right)\right)= -   \dfrac{1}{P^2\vert\nabla_{g}\xi\vert^2}\dfrac{d}{d\xi}\left(P^2\vert\nabla_{g}\xi\vert^2\right).
		\end{eqnarray*}
		By integration, 
		\begin{eqnarray*}
			\dfrac{d}{d\xi}\left(-N^{-1}\right)=\frac{k}{P^2\vert\nabla_{g}\xi\vert^2},
		\end{eqnarray*}
		where $k$ is a non-null constant. So, 
		\begin{eqnarray}\label{eq2-thm3}
			- N^{-1}(\xi)=k\int \dfrac{1}{P^2\vert\nabla_{g}\xi\vert^2}d\xi + k_1,
		\end{eqnarray}
		where $k_1\in\mathbb{R}$.

		Note that $P^2\vert\nabla_{g}\xi\vert^2=\sum_{k=1}^{m_1}\left(a_k-b_k\xi\right)^2+\sum_{k=m_1+1}^{m_2}b_k^2\xi^2$, which can be rewritten as, $P^2\vert\nabla_{g}\xi\vert^2=\eta\xi^2+\theta\xi+\delta$ with $\eta=\sum_{k=1}^{m_2}b_k^2$, $\theta=-2\sum_{k=1}^{m_1}a_kb_k$, and $\delta=\sum_{k=1}^{m_1}a_k^2$. Thus,
		\begin{eqnarray*}
			\int \frac{1}{P^2\vert\nabla_{g}\xi\vert^2}d\xi=\int\frac{1}{\eta\xi^2+\theta\xi+\delta}d\xi.
		\end{eqnarray*}
		
		We know that $4\eta\delta-\theta^2 \ge 0$ (see Remark~\ref{rem:lapse-positive}).
		If $4\eta\delta-\theta^2 > 0$, the integral in~\eqref{eq2-thm3} can be computed explicitly, yielding
		\[
		-\,N^{-1}(\xi)
		=
		k_{1}
		+
		\frac{2k}{\sqrt{4\eta\delta - \theta^{2}}}
		\arctan\!\left(
		\frac{2\eta\xi + \theta}{\sqrt{4\eta\delta - \theta^{2}}}
		\right),
		\]
		where $4\eta\delta - \theta^{2}>0$.
	\end{proof}


	\

	\noindent{\bf Conflict of Interest:} There is no conflict of interest to disclose.
	
	\
	
	\noindent{\bf Data Availability:} Not applicable.
	
	\

\end{document}